\documentclass[reqno,12pt]{amsart}
\usepackage{amsmath,amsthm,amssymb,amsfonts,amscd}
\usepackage{epsfig}
\usepackage{color}
\usepackage[all]{xy}
\usepackage{tikz}
\theoremstyle{plain}
\newtheorem{theorem}{Theorem}

\newtheorem{lemma}[theorem]{Lemma}

\newtheorem{proposition}[theorem]{Proposition}
\newtheorem*{theorem*}{Theorem}
\newtheorem*{conjecture*}{Conjecture}

\theoremstyle{definition}
\newtheorem{remark}[theorem]{Remark}

\newtheorem{definition}[theorem]{Definition}

\usepackage{graphics, setspace}

\newcommand{\breakingcomma}{%
  \begingroup\lccode`~=`,
  \lowercase{\endgroup\expandafter\def\expandafter~\expandafter{~\penalty0 }}}
  
\usepackage{breqn}

\usepackage{amstext} 
\usepackage{array}
\newcolumntype{L}{>{$}l<{$}}

\allowdisplaybreaks

\newcommand{\CC}{{\mathbb{C}}}

\newcommand{\QQ}{{\mathbb{Q}}}

\newcommand{\ZZ}{{\mathbb{Z}}}

\newcommand{\SL}{{\mathrm{SL}}}
\newcommand{\epi}{{\bf e}}

\newcommand{\Jac}{\mathrm{Jac}}
\newcommand{\Fix}{\mathrm{Fix}}

\newcommand{\Aut}{\mathrm{Aut}}
\newcommand{\id}{\mathrm{id}}

\newcommand{\bx}{{\bf x}}
\newcommand{\by}{{\bf y}}
\newcommand{\bz}{{\bf z}}
\newcommand{\hess}{\mathrm{hess}}
\newcommand{\age}{\mathrm{age}}

\newcommand{\QR}[2]{
\left.\raisebox{0.5ex}{\ensuremath{#1}}
\ensuremath{\mkern-3mu}\middle/\ensuremath{\mkern-3mu}
\raisebox{-0.5ex}{\ensuremath{#2}}\right.}

\def\A{{\mathcal A}}

\def\MF{{\mathrm{MF}}}

\def\Cl{{\mathrm{Cl}}}

\newcommand{\rmH}{{{\rm H}}}

\newcommand{\LL}{{\Upsilon}}

\newcommand{\ccHH}{{\mathsf{HH}}}

\begin{document}
\title[Hochschild cohomology and orbifold Jacobian algebras]{Hochschild cohomology and orbifold Jacobian algebras 
associated to invertible polynomials}
\date{\today}
\author{Alexey Basalaev}
\address{Skolkovo Institute of Science and Technology, Russian Federation}
\email{a.basalaev@skoltech.ru}
\address{Ruprecht-Karls-Universit\"at Heidelberg, Germany}
\email{abasalaev@mathi.uni-heidelberg.de}
\author{Atsushi Takahashi}
\address{Department of Mathematics, Graduate School of Science, Osaka University, 
Toyonaka Osaka, 560-0043, Japan}
\email{takahashi@math.sci.osaka-u.ac.jp}
\begin{abstract}
Let $f$ be an invertible polynomial and $G$ a group of diagonal symmetries of $f$.
This note shows that the orbifold Jacobian algebra $\Jac(f,G)$ of $(f,G)$ defined by \cite{BTW16} 
is isomorphic as a $\ZZ/2\ZZ$-graded algebra to the Hochschild cohomology $\ccHH^*(\MF_G(f))$ of the dg-category $\MF_G(f)$
of $G$-equivariant matrix factorizations of $f$ by calculating the product formula of $\ccHH^*(\MF_G(f))$
given by Shklyarov~\cite{S17}.
We also discuss the relation of our previous results to the categorical equivalence.
\end{abstract}
\maketitle

\setcounter{tocdepth}{1}

\section{Introduction}
To a polynomial $f=f(x_1,\dots, x_n)\in\CC[x_1,\dots, x_n]$ which defines an isolated singularity only at the origin, 
one can associate a finite dimensional $\CC$-algebra called 
the Jacobian algebra $\Jac(f):= \CC[x_1,\dots, x_n]/(\partial f/\partial x_1,\dots,\partial f/\partial x_n)$ of $f$, 
which is an important and interesting invariant of $f$.
 
In our previous paper \cite{BTW16}, 
when $f$ is an invertible polynomial and $G$ is a finite abelian group acting diagonally on variables which respects $f$
we show the existence and the uniqueness of the $G$-twisted version of the Jacobian algebra of $f$, 
a $\ZZ/2\ZZ$-graded $G$-twisted commutative algebra denoted by $\Jac'(f,G)$.
The expected Jacobian algebra $\Jac(f,G)$ of the pair $(f,G)$, 
a natural generalization of the Jacobian algebra $\Jac(f)$ of $f$ to the pair $(f,G)$,
is then given as the $G$-invariant subalgebra of $\Jac'(f,G)$ and it is called the orbifold Jacobian algebra of $(f,G)$. 

Another important invariant associated to the pair $(f,G)$ is the dg-category $\MF_G(f)$ of $G$-equivariant matrix factorizations of $f$. To this category, one can associate the Hochschild cohomology $\ccHH^*(\MF_G(f))$, which is equipped 
with a $\ZZ/2\ZZ$-graded commutative cup product.
Actually, our axiomatization of $\Jac'(f,G)$ and $\Jac(f,G)$ is motivated by the algebraic structure of 
the pair $\left(\ccHH^*(\MF_G(f)), \ccHH_*(\MF_G(f))\right)$.
It is natural to show now that $\Jac(f,G)$ is isomorphic to $\ccHH^*(\MF_G(f))$.

Recently, Shklyarov \cite{S17} developed a method to compute $\ccHH^*(\MF_G(f))$. 
He introduces a $\ZZ/2\ZZ$-graded $G$-twisted commutative algebra $\A^*(f,G)$ 
such that its $G$-invariant part $\A^*(f,G)^G$ is isomorphic as a $\ZZ/2\ZZ$-graded algebra 
to the Hochschild cohomology $\ccHH^*(\MF_G(f))$ \cite[Theorem~3.1 and Theorem~3.4]{S17}. 

It is natural to expect the isomorphism between $\Jac'(f,G)$ and $\A^*(f,G)$ 
since they have the same underlying $\ZZ/2\ZZ$-graded $\CC$-vector spaces with similar algebraic structures. 
We show this by calculating the Shklyarov's product formula of $\A^*(f,G)$.
\begin{theorem*}[Theorem~\ref{theorem: main}]
Let $f$ be an invertible polynomial and $G$ a subgroup of the group $G_f$ 
of maximal diagonal symmetries. We have a $\ZZ/2\ZZ$-graded algebra isomorphism
\begin{equation}
\Jac'(f,G)\cong \A^*(f,G),
\end{equation}
which is compatible with the $G$-actions on both sides. 
In particular, by taking the $G$-invariant part, we have a $\ZZ/2\ZZ$-graded algebra isomorphism
\begin{equation}
\Jac(f,G)\cong \ccHH^*(\MF_G(f)).
\end{equation}
\end{theorem*}
When $f$ is of chain type in two variables, this is done in \cite[Appendix~A.1]{S17}. 
\begin{remark}
While preparing the manuscript, a closely related work by He--Li--Li \cite{HLL} has appeared.
They seem to give another method to calculate the cup product of 
the Hochschild cohomology $\ccHH^*(\CC[\bx]\rtimes  G,f)$ of $G$-equivariant curved algebra
(cf. \cite[Section~2]{S17}) isomorphic to $\A^*(f,G)$. Their product formula coincides (up to sign) with 
the Shklyarov's formula.
\end{remark}
\bigskip

\noindent{\bf Acknowledgements.}
The first named author is very grateful to Dmytro Shklyarov for sharing his ideas and also the draft version of his text \cite{S17}. We are grateful to Elisabeth Werner for the help on the early stages of the project and to Wolfgang Ebeling for fruitful discussions.
The second named author is supported by JSPS KAKENHI Grant Number JP16H06337.

\section{Notations and terminologies}\label{sec:notations}
For a non-negative integer $n$ and $f=f(x_1,\dots, x_n)\in\CC[x_1,\dots, x_n]$ a polynomial, 
the {\em Jacobian algebra} $\Jac(f)$ of $f$ is a $\CC$-algebra defined as 
\begin{equation}
\Jac(f)= \QR{\CC[x_1,\dots, x_n]}{\left(\frac{\partial f}{\partial x_1},\dots,\frac{\partial f}{\partial x_n}\right)}.
\end{equation}
If $\Jac(f)$ is a finite-dimensional, then set $\mu_f:=\dim_\CC\Jac(f)$ and call it the {\em Milnor number} of $f$. 
In particular, if $n=0$ then $\Jac(f)=\CC$ and $\mu_f=1$.
The {\em Hessian} of $f$ is defined as 
\begin{equation}
\hess(f):=\det \left(\frac{\partial^2 f}{\partial x_{i} \partial x_{j}}\right)_{i,j=1,\dots,n}.
\end{equation}
In particular, if $n=0$ then $\hess(f)=1$.

A polynomial $f\in\CC[x_1,\dots, x_n]$ is called a {\em weighted homogeneous} polynomial  
if there are positive integers $w_1,\dots ,w_n$ and $d$ such that 
$f(\lambda^{w_1} x_1, \dots, \lambda^{w_n} x_n) = \lambda^d f(x_1,\dots ,x_n)$ 
for all $\lambda \in \CC^\ast$.
A weighted homogeneous polynomial $f$ is called {\em non-degenerate}
if it has at most an isolated critical point at the origin in $\CC^n$, equivalently, if the Jacobian algebra $\Jac(f)$ of $f$ 
is finite-dimensional.
\begin{definition}\label{def:invertible}
A non-degenerate weighted homogeneous polynomial $f\in\CC[x_1,\dots, x_n]$ is called {\em invertible} if 
the following conditions are satisfied.
\begin{itemize}
\item 
The number of variables coincides with the number of monomials in $f$: 
\begin{equation}
f(x_1,\dots ,x_n)=\sum_{i=1}^n c_i\prod_{j=1}^n x_j^{E_{ij}}
\end{equation}
for some coefficients $c_i\in\CC^\ast$ and non-negative integers 
$E_{ij}$ for $i,j=1,\dots, n$.
\item
The matrix $E:=(E_{ij})$ is invertible over $\QQ$.
\end{itemize}
\end{definition}
Let $f=\sum_{i=1}^n c_i\prod_{j=1}^n x_j^{E_{ij}}$ be an invertible polynomial.
Without loss of generality one may assume that $c_i=1$ for all $i$ by rescaling the variables. 
According to \cite{KS}, an invertible polynomial $f$ can be written as a Thom--Sebastiani sum
$f=f_1\oplus\cdots\oplus f_p$ of invertible ones (in groups of different variables) 
$f_\nu$, $\nu=1,\dots, p$ of the following types:
\begin{enumerate}
\item $x_1^{a_1}x_2+x_2^{a_2}x_3+\dots+x_{m-1}^{a_{m-1}}x_m+x_m^{a_m}$ ({\em chain type}, $m \geq 1$);
\item  $x_1^{a_1}x_2+x_2^{a_2}x_3+\dots+x_{m-1}^{a_{m-1}}x_m+x_m^{a_m}x_1$ ({\em loop type}, $m \geq 2$).
\end{enumerate}
\begin{remark}
In \cite{KS} the authors distinguished also polynomials of the so called Fermat type: $x_1^{a_1}$, 
which is regarded as a chain type polynomial with $m = 1$ in this paper.
\end{remark}
\begin{definition}
The {\em group of maximal diagonal symmetries} of an invertible polynomial $f(x_1,\dots, x_n)$ is defined as 
\begin{equation}
G_f := \left\{ (\lambda_1, \ldots , \lambda_n )\in(\CC^\ast)^n \, \left| \, f(\lambda_1x_1, \ldots, \lambda_n x_n) = 
f(x_1, \ldots, x_n) \right\} \right..
\end{equation}

Each element $g\in G_f$ has a unique expression of the form
$g=(\epi[\alpha_1], \dots,\epi[\alpha_n])$ with $0 \leq \alpha_i < 1$,
where $\epi\left[\alpha\right] := \exp(2 \pi \sqrt{-1} \alpha)$.
The {\em age} of $g$ is defined as the rational number ${\rm age}(g) := \sum_{i=1}^n \alpha_i$. 
\end{definition}

For each $g\in G_f$, let $I_g := \{i_1,\dots,i_{n_g}\}$ be a subset of $\{1,\dots, n\}$ such that $\Fix(g)=\{x\in\CC^n~|~x_{j}=0, j\notin I_g\}$.
In particular, $I_{\id}=\{1,\dots, n\}$ and $n_g=\dim_\CC\Fix(g)$. Denote by $I_g^c$ the complement of $I_g$ in $I_\id$ and 
set $d_g:=n-n_g$, the codimension of $\Fix(g)$.
\begin{proposition}\label{prop:I}
For $f = x_1^{a_1}x_2 + \dots + x_{n-1}^{a_{n-1}}x_n + x_n^{a_n}$ of chain type,  
for each $g \in G_f\backslash\{\id\}$ there exists $1\le k \le n$, such that $I_g^c = \{1,\dots,k\}$. 
For $f = x_1^{a_1}x_2 + \dots + x_n^{a_n}x_1$ of loop type, 
for each element $g \in G_f\backslash\{\id\}$ has $I_g^c = \{1,\dots,n\}$.
\end{proposition}
\section{Orbifold Jacobian algebra}
We briefly recall our orbifold Jacobian algebras. 
>From now on, let $f=f(x_1,\dots, x_n)$ be an invertible polynomial and $G$ a subgroup of $G_f$. 
In order to simplify the notation, we often write $f(x_1,\dots, x_n)$ as $f(\bx)$, $\CC[x_1,\dots, x_n]$ as $\CC[\bx]$ and so on.

{\it A $G$-twisted Jacobian algebra $\Jac'(f,G)$ of $f$} is a $\ZZ/2\ZZ$-graded algebra characterized by a set of axioms \cite[Section~3]{BTW16} motivated by properties satisfied by the pair of the Hochschild cohomology and homology. 
It exists and it is uniquely defined up to an isomorphism \cite[Theorem~21]{BTW16}. 
This allows us to describe it by the explicit formula. 

As a $G \times \ZZ/2\ZZ$-graded $\CC$-vector space, we have
\begin{equation}\label{eq: JacFG as v.sp.}
\Jac'(f,G) = \bigoplus_{g \in G} \Jac(f^g)\widetilde v_g, \quad f^g := f|_{\Fix(g)},
\end{equation}
where $\widetilde v_g$ is a generator (a formal letter) attached to each $g\in G$. 
The $\ZZ/2\ZZ$-grading of the element $[\phi(\bx)] \tilde v_g$ is defined by the parity of $d_g$. 
$\Jac'(f,G)$ is endowed with a $\ZZ/2\ZZ$-graded $G$-twisted commutative\footnote{in \cite{S17} it is called \textit{braided}  commutative} product $\circ$ with the unit $\widetilde v_\id=[1]\widetilde v_\id$. Namely, for all $g,h\in G$
\begin{align}
  [\psi(\bx)] \widetilde v_g & \circ [\phi(\bx)] \widetilde v_h \in \Jac(f^{gh}) \widetilde v_{gh},
  \\
  [\psi(\bx)] \widetilde v_g & \circ [1] \widetilde v_\id =[\psi(\bx)]\widetilde v_{g}=[1] \widetilde v_\id \circ[\psi(\bx)]\widetilde v_{g}, 
\end{align}
and
\begin{equation}\label{eq: twisted commutativity}
[\psi(\bx)] \widetilde v_g \circ [\phi(\bx)] \widetilde v_h = (-1)^{d_gd_h} \cdot g^*([\phi(\bx)] \widetilde v_h) \circ [\psi(\bx)] \widetilde v_g
\end{equation}
where $g^*$ is the $G$-action, given on the subspace $\Jac(f^h)\widetilde v_h$ by 
\begin{equation}\label{G-action1}
g^*([\phi(\bx)] \widetilde v_h)= \prod_{i\in I_h^c} g^{-1}_i\,\cdot [\phi(g\cdot\bx)]\widetilde v_h,\quad g=(g_1,\dots ,g_n).
\end{equation}
Note that $f^g$ is also an invertible polynomial 
and there is a surjective map $\Jac(f)\longrightarrow \Jac(f^g)$ (\cite[Proposition~5]{ET13} and \cite[Proposition~7]{BTW16}). 
The product is compatible with this map and, in particular, we have  $[\psi(\bx)]\widetilde v_\id \circ [\phi(\bx)]\widetilde v_g=[\psi(\bx)\phi(\bx)]\widetilde v_g$. 

In order to explain the explicit product formula for $\Jac'(f,G)$ simpler, we give 
the following K\"unneth property of $\Jac'(f,G)$.
\begin{proposition}
Let $f_1$, $f_2$ be invertible polynomials 
and $G_1\subseteq G_{f_1}$, $G_2\subseteq G_{f_2}$ be subgroups. We have an algebra isomorphism compatible with 
$(G_1\times G_2)$-actions on both sides:
\begin{equation}\label{eq: kuenneth in Jac'}
\Jac'(f_1\oplus f_2,G_1\times G_2) \cong \Jac'(f_1,G_1) \otimes_\CC \Jac'(f_2,G_2).
\end{equation}
\end{proposition}
\begin{proof}
This is a direct consequence of a set of axioms \cite[Section~3]{BTW16} or the product formula \cite[Corollary~43]{BTW16}.
The key fact is that for any $g_1\in G_1$ and $g_2\in G_2$ we have
$\tilde v_{g_1} \circ \tilde v_{g_2} := \mathrm{sgn}(\sigma_{g_1,g_2}) \tilde v_{g_1g_2}$
where $\sigma_{g_1,g_2}$ is the permutation that turns the ordered sequence $I_{g_1}^c\sqcup I_{g_2}^c$ to $I_{g_1g_2}^c$. In \cite{BTW16}, $\mathrm{sgn}(\sigma_{g_1,g_2})$ is denoted by $\widetilde\epsilon_{g_1,g_2}$
\end{proof}

By this K\"unneth property, we assume $f$ is of chain type or of loop type.
For each pair $(g,h)$ of elements in $G$ and $\phi(\bx), \psi(\bx) \in \CC[\bx]$, the product formula is given as follows.
\begin{itemize}
\item
If $gh\ne \id$, $g\ne \id$ and $h\ne \id$, then $[\phi(\bx)]\widetilde v_g\circ [\psi(\bx)]\widetilde v_h =0$.
\item 
If $gh=\id$, then 
\begin{equation}\label{prod str}
[\phi(\bx)]\widetilde v_g\circ [\psi(\bx)]\widetilde v_{g^{-1}} 
=(-1)^{\frac{d_g(d_g-1)}{2}}\cdot \epi\left[-\frac{1}{2}\age(g)\right]\cdot  [\phi(\bx)\psi(\bx) H_{g,g^{-1}}]\widetilde v_{\id},
\end{equation}
where 
\begin{equation}
H_{g,g^{-1}}:=\widetilde{m}_{g,g^{-1}}\det\left(\frac{\partial^2 f}{\partial x_i\partial x_j}\right)_{i,j\in I_g^c}
\end{equation}
and $\widetilde{m}_{g,g^{-1}}$ is a constant uniquely determined by the following equation in $\Jac(f)$
\begin{equation}\label{H and hessians}
\frac{1}{\mu_{f^g}}[\hess(f^g)H_{g,g^{-1}}]=\frac{1}{\mu_{f}}[\hess(f)].
\end{equation}
\end{itemize}

The product $\circ$ is invariant under the $G$-action of \eqref{G-action1} (cf. \cite[Proposition 58]{BTW16}) and 
hence $G$-invariant subspace $\left( \Jac'(f,G) \right)^G$ is a $\ZZ/2\ZZ$-graded commutative algebra.

\begin{definition}
Let $f$ and $G$ be as above. The $G$-invariant $\ZZ/2\ZZ$-graded subalgebra $\Jac(f,G) := \left( \Jac'(f,G) \right)^G$ is called 
the {\em orbifold Jacobian algebra} of $(f,G)$. 
\end{definition}

\section{Shklyarov's description of Hochschild cohomology}\label{sec:Shk}
Shklyarov \cite{S17} introduces a $\ZZ/2\ZZ$-graded $G$-twisted commutative algebra $\A^*(f,G)$ 
whose underlying  $\CC$-vector space is given by 
\begin{equation}
\A^*(f,G) = \bigoplus_{g \in G} \Jac(f^g)\xi_g,
\end{equation}
where $\widetilde\xi_g$ is a generator (a formal letter) attached to each $g\in G$.
It is required that the group $G$ acts naturally on $\Jac(f^g)$ for each $g \in G$ and $\xi_g$ transforms as
\begin{equation}\label{G-action2}
G\ni h=(h_1, \ldots, h_n): \quad \widetilde \xi_g\mapsto \prod_{i\in I_g^c} h^{-1}_i\,\cdot  \widetilde \xi_g.
\end{equation}
so that the product structure of $\A^*(f,G)$ is invariant under the $G$-action.
In particular, its $G$-invariant part $\A^*(f,G)^G$ is isomorphic as a $\ZZ/2\ZZ$-graded algebra 
to the Hochschild cohomology $\ccHH^*(\MF_G(f))$ of $\MF_G(f)$ equipped with the cup product
\cite[Theorem~3.1 and Theorem~3.4]{S17}. 
We shall recall the product structure of $\A^*(f,G)$ \cite[Section~3]{S17}. 

Define the $n$-th $\ZZ$-graded {\em Clifford algebra} $\Cl_n$ as the quotient algebra of 
\[
\CC\langle\theta_1,\ldots,\theta_n,{\partial_{\theta_1}},\ldots,{\partial_{\theta_n}}\rangle
\]
modulo the ideal generated by 
\[
\theta_i\theta_j=-\theta_j\theta_i,\quad{\partial_{\theta_i}}{\partial_{\theta_j}}=-{\partial_{\theta_j}}{\partial_{\theta_i}},\quad{\partial_{\theta_i}}\theta_j=-\theta_j{\partial_{\theta_i}}+\delta_{ij},
\]
where $\theta_i$ is of degree $-1$ and ${\partial_{\theta_i}}$ is of degree $1$.
For $I\subseteq\{1,\ldots,n\}$ write
\begin{equation}
\partial_{\theta_{I}}:=\prod_{i\in I}\partial_{\theta_i}, 
\quad {\theta_{I}}:=\prod_{i\in I}{\theta_i},
\end{equation}
where in both cases  the multipliers are taken in increasing order of the indices. 
The subspaces $\CC[{\theta}]=\CC[\theta_1,\ldots,\theta_n]$ and $\CC[\partial_{\theta}]=\CC[\partial_{\theta_1},\ldots,\partial_{\theta_n}]$ of $\Cl_n$ have the left $\ZZ$-graded $\Cl_n$-module structures via the isomorphisms  
\[
\CC[{\theta}]\cong \Cl_n/\Cl_n\langle{\partial_{\theta_1}},\ldots,{\partial_{\theta_n}}\rangle, \quad \CC[\partial_{\theta}]\cong \Cl_n/\Cl_n\langle{{\theta_1}},\ldots,{{\theta_n}}\rangle. 
\]  

Write $\CC[x_1,\dots,x_n,y_1,\ldots, y_n]$ as $\CC[\bx,\by]$ and 
$\CC[x_1,\dots,x_n,y_1,\ldots, y_n, z_1,\dots, z_n]$ as $\CC[\bx,\by,\bz]$. 
For each $1 \le i \le n$, there is a map
\begin{eqnarray}\label{deltai}
\nabla^{\bx\to (\bx,\by)}_i: \CC[\bx]\to \CC[\bx,\by],\qquad \nabla_i(p):=\frac{l_i(p)-l_{i+1}(p)}{x_i-y_i}.
\end{eqnarray}
where $l_i(p):=p(y_1,\ldots,y_{i-1},x_i,\ldots,x_n)$, $l_1(p) = p(\bx)$ and $l_{n+1}(p) = p(\by)$ \cite[Section~3.1.1]{S17}. 
They are called the \textit{difference derivatives}, whose key property is the following: 
\begin{equation}\label{diffder}
\sum_{i=1}^n(x_i-y_i)\nabla_i(p)=p(\bx)-p(\by).
\end{equation}

The difference derivatives can be applied consecutively. In particular, we shall use 
$\nabla_i^{\by\to(\by,\bz)}\nabla_j^{\bx\to(\bx,\by)}(p)$, which is an element of $\CC[\bx,\by,\bz]$.  
For an $\CC$-algebra homomorphism $\psi: \CC[\bx] \to \CC[\bx]$, write $\nabla_i^{\bx\to(\bx,\psi(\bx))}(p) := \left.\nabla_i^{\bx\to(\bx,\by)}(p) \right|_{\by = \psi(\bx)} \in \CC[\bx]$.

Now we are ready to describe the product structure of $\A^*(f,G)$. For each pair $(g,h)$ of elements in $G$, define the class $\sigma_{g,h}\in \Jac(f^{gh})$ as follows.
\begin{itemize}
\item
If $d_{g,h}:=\frac{1}{2}(d_g +d_h - d_{gh})$ is not a non-negative integer, set $\sigma_{g,h} = 0$.
\item
If $d_{g,h}$ is a non-negative integer, define $\sigma_{g,h}$ to be the class of the coefficient of $\partial_{\theta_{I_{gh}^c}}$ in the expression 
\begin{equation}\label{eq: sigma g,h}
\Small{
\frac1{d_{g,h}!}\,\LL\left(\left(\left\lfloor\rmH_f(\bx,g(\bx),\bx)\right\rfloor_{gh}+\lfloor\rmH_{{f,g}}(\bx)\rfloor_{gh}\otimes1+1\otimes \lfloor \rmH_{{f,h}}(g(\bx))\rfloor_{gh}\right)^{d_{g,h}}\otimes \partial_{\theta_{I_g^c}}\otimes\partial_{\theta_{h}^c}\right)
}
\end{equation} 
where 
\begin{itemize}
\item[(1)]
$\rmH_f(\bx,g(\bx),\bx)$ is the element of $\CC[\bx]\otimes\CC[\theta]^{\otimes2}$ defined as the restriction to the set 
$\{\by=g(\bx),\, \bz=\bx\}$ of the following element of  $\CC[\bx,\by,\bz]\otimes\CC[\theta]^{\otimes 2}$ 
\begin{equation}\label{del-2}
\rmH_{f}(\bx,\by,\bz):=\sum_{1\leq j\leq i\leq n} \nabla^{\by\to(\by,\bz)}_j\nabla^{\bx\to(\bx,\by)}_i(f)\,\theta_i\otimes \theta_j;
\end{equation}
\item[(2)]
$\rmH_{{f,g}}(\bx)$ is the element of $\CC[\bx]\otimes \CC[\theta]$ is given by 
\begin{eqnarray}
\rmH_{{f,g}}(\bx):=\sum_{{i,j\in I_{g}^c,\,\,  j<i}}\frac{1}{1-g_j}\nabla^{\bx\to(\bx,\bx^g)}_j\nabla^{\bx\to(\bx,g(\bx))}_i(f)\,\theta_j\,\theta_i,
\end{eqnarray}
where $\bx^g$ is defined as $(\bx^g)_i=x_i$ if $i\in I_g$ and $(\bx^g)_i=0$ if $i\in I_g^c$;
\item[(3)]
$\left\lfloor\rm- \right\rfloor_{gh}:\CC[\bx]\otimes V\longrightarrow \Jac(f^{gh})\otimes V$ for 
$V=\CC[\bx]\otimes \CC[\theta]^{\otimes 2}$ or $V=\CC[\bx]\otimes \CC[\theta]$
is a $\CC$-linear map defined as the extension of the quotient map $\CC[\bx]\longrightarrow \Jac(f^{gh})$;
\item[(4)]
the $d_{g,h}$-th power  in Equation~\eqref{eq: sigma g,h} is computed with respect to the natural product on  $\CC[\bx]\otimes \CC[\theta]\otimes \CC[\theta]$;
\item[(5)]
$\LL$ is the $\CC[\bx]$-linear extension of the degree zero map $\CC[{\theta}]^{\otimes2}\otimes \CC[\partial_{\theta}]^{\otimes2}\to\CC[\partial_{\theta}]$ defined by
\begin{equation}\label{mu}
 p_1(\theta)\otimes p_2(\theta)\otimes q_1(\partial_\theta)\otimes q_2(\partial_\theta)\mapsto(-1)^{|q_1||p_2|}p_1(q_1)\cdot p_2(q_2)
\end{equation}
where $p_i(q_i)$ denotes the action of $p_i(\theta)$ on $q_i(\partial_\theta)$ via the $\Cl_n$-module structure on $\CC[\partial_\theta]$ defined above and $\cdot$ is the natural product in $\CC[\partial_{\theta}]$. 
\end{itemize}
\end{itemize}
Then the product of $\A^*(f,G)$ is given by
\begin{equation}\label{eq: HH cup}
[\phi(\bx)]\xi_g\cup [\psi(\bx)]\xi_{h} 
=[\phi(\bx)\psi(\bx) \sigma_{g,h}]\xi_{gh},\quad \phi(\bx), \psi(\bx)\in \CC[\bx].
\end{equation}
\section{Results}
Now we can state our main theorem in this paper.
\begin{theorem}\label{theorem: main}
Let $f$ be an invertible polynomial and $G$ a subgroup of $G_f$. We have a $\ZZ/2\ZZ$-graded algebra isomorphism
$\Jac'(f,G)\cong \A^*(f,G)$ compatible with the $G$-actions on both sides. 
\end{theorem}

One sees immediately that $\Jac'(f,G)$ and $\A^*(f,G)$ have isomorphic underlying $\ZZ/2\ZZ$-graded vector spaces. 
It is also clear that the $G$-actions are compatible on both sides (recall \eqref{G-action1} and \eqref{G-action2}). 
It remains to check that the products agree.
One only needs to do this for $f$ of chain type or of loop type since $\Jac'(f,G)$ and $\A^*(f,G)$ have the same K\"unneth property \eqref{eq: kuenneth in Jac'} and \cite[Proposition~2.6]{S17}.
The rest of this section is devoted to the proof of the theorem. 

\begin{proof}
First, similarly to the case $\Jac'(f,G)$, we have the following
\begin{proposition}
If $gh\ne \id$, $g\ne \id$ and $h\ne \id$, then $\sigma_{g,h}=0$.
\end{proposition}
\begin{proof}
Since the algebra $\A^*(f,G)$ is also $\Aut(f,G)$-invariant due to \cite[Theorem~3.1]{S17}, 
we may apply \cite[Proposition~34]{BTW16}, which yields the statement. 
We can also show the vanishing of $\sigma_{g,h}$ due to degree reason by direct calculation, which is elementary.
\end{proof}
Hence, we only need to calculate $\sigma_{g,g^{-1}}$ for each $g\in G_f\backslash\{\id\}$.
\begin{proposition}\label{prop:10}
For each $g=(\epi[\alpha_1], \dots,\epi[\alpha_n])\in G_f\backslash\{\id\}$ with $0\le \alpha_i<1$, we have 
\begin{equation}
\sigma_{g,g^{-1}}=(-1)^{\frac{d_g(d_g-1)}{2}}\cdot \epi\left[-\frac{1}{2}\age(g)\right]\cdot 
\left(\prod_{i=1}^{d_g}\frac{\epi\left[-\frac{1}{4}\right]}{2\sin(\alpha_i\pi)}\right) [H_{g,g^{-1}}].
\end{equation}
\end{proposition}
\begin{proof}
For $f$ of chain type one applies Lemma~\ref{lem:11} and Lemma~\ref{lem:12} (see Section~\ref{sec:chain} below) and 
for $f$ of loop type one does Lemma~\ref{lem:13} and Lemma~\ref{lem:14} (see Section~\ref{sec:loop} below).
\end{proof}
Note that for $i=1,\dots, d_g$ we have $0< \alpha_i<1$ 
and that $\prod_{i=1}^{d_g}2\sin(\alpha_i\pi)$ is invariant under taking $g$ 
to its inverse $g^{-1}$ since this is equivalent to substituting $\alpha_i$ with $1-\alpha_i$. 

The algebra isomorphism $\A^*(f,G) \to \Jac'(f,G)$ reads:
\begin{equation}
  \tilde v_{g} \to \epi\left[-\frac{d_g}{8}\right] \left(\prod_{i=1}^{d_g} 2\sin(\alpha_i\pi) \right)^{-1/2} \xi_g.
\end{equation}
\end{proof}


\subsection{Chain type $f = x_1^{a_1}x_2 + \dots + x_{n-1}^{a_{n-1}}x_n + x_n^{a_n}$}\label{sec:chain}
First, we list the ingredients for the product formula \eqref{eq: HH cup}:
{\small 
\[
\rmH_f(\bx,\by,\bz) 
= \sum_{i=1}^n \frac{x_{i+1}}{y_i-z_i} \left( \frac{x_i^{a_i}-y_i^{a_i}}{x_i-y_i} - \frac{x_i^{a_i}-z_i^{a_i}}{x_i-z_i} \right)\theta_i \otimes\theta_i
  + \sum_{i=1}^{n-1} \frac{y_i^{a_i}-z_i^{a_i}}{y_i-z_i} \theta_{i+1} \otimes\theta_{i}, \quad x_{n+1}:=1,
\] 
\[
\rmH_{f,g}(\bx) = \sum_{i=1}^{d_g-1} x_i^{a_i-1} \frac{g_i^{a_i}}{1 - g_i} \theta_i \theta_{i+1},\quad g=(g_1,\dots,g_n)\in G_f\backslash\{\id\}. 
\]}
They give
{\small 
\[
\rmH_f(\bx,g(\bx),\bx) 
= \sum_{i=1}^n \frac{x_i^{a_i-2}x_{i+1}}{g_i-1} \left( \frac{1-g_i^{a_i}}{1-g_i} - a_i \right)\theta_i \otimes\theta_i
  + \sum_{i=1}^{n-1} \frac{g_i^{a_i}-1}{g_i-1}x_i^{a_i-1} \theta_{i+1} \otimes\theta_{i},
\]
\[
\rmH_{f,g^{-1}}(g(\bx)) = \sum_{i=1}^{d_g-1} x_i^{a_i-1} \frac{1}{g_i - 1} \theta_i \theta_{i+1},\quad g=(g_1,\dots,g_n)\in G_f\backslash\{\id\}.
\]}
Consider a square matrix $M_g$ of size $d_g$ whose $(i,j)$-th entry is given by
{\small 
\begin{equation}
(M_g)_{ij} := -\frac{x_i^{a_i-2}x_{i+1}}{g_i-1} \left( \frac{1-g_i^{a_i}}{1-g_i} - a_i \right)\delta_{i,j} 
+ x_i^{a_i-1} \frac{g_i^{a_i}}{1 - g_i} \delta_{i+1,j}
+ x_{i-1}^{a_{i-1}-1} \frac{1}{g_{i-1} - 1}\delta_{i-1,j}.
\end{equation}}
\begin{lemma}\label{lem:11}
For each $g\in G_f\backslash\{\id\}$, we have the following equality in $\Jac(f)$:
\begin{equation}
\sigma_{g,g^{-1}}=(-1)^{\frac{d_g(d_g-1)}{2}}\left[\det(M_{g})\right].
\end{equation}
\end{lemma}
\begin{proof}
Since $I_{\id}^c = \emptyset$ and due to \eqref{mu}, 
we only consider the coefficient of $\theta_{I_g^c}\otimes \theta_{I_g^c}$ in  
\[
{\small
\frac{1}{d_g!}\left(\left\lfloor\rmH_f(\bx,g(\bx),\bx)\right\rfloor_{\id}+\lfloor\rmH_{{f,g}}(\bx)\rfloor_{\id}\otimes1+1\otimes \lfloor \rmH_{{f,h}}(g(\bx))\rfloor_{\id}\right)^{d_{g}}.
}\]
It is clear from this and also the specific form of $\rmH_f(\bx,g(\bx),\bx)$, $\rmH_{f,g}(\bx)$ and $\rmH_{f,g^{-1}}(g(\bx))$ that 
the coefficients of $\theta_{i+1} \otimes\theta_{i}$ in $\rmH_f(\bx,g(\bx),\bx)$ do not contribute to $\sigma_{g,g^{-1}}$. Due to the same reason one sees that $\rmH_{f,g}$ and $\rmH_{f,g^{-1}}$ contribute to $\sigma_{g,g^{-1}}$ in pairs.
Taking care of the Clifford algebra coefficients, one obtains a recurrence formula for $\sigma_{g,g^{-1}}$ in terms of
$\det(M_g)$ via its minors.
\end{proof}
\begin{lemma}\label{lem:12}
For each $g\in G_f\backslash\{\id\}$, we have 
{\small 
\begin{equation}
\left[\det(M_{g})\right]=\left(\prod_{i=1}^{d_g}\frac{-a_i}{1-g_i}\right)[x_1^{a_1-2}x_2^{a_2-1}\cdots x_{d_g}^{a_{d_g}-1}x_{d_g+1}]\\
=\left(\prod_{i=1}^{d_g}\frac{1}{g_i-1}\right) [H_{g,g^{-1}}].
\end{equation}}
\end{lemma}
\begin{proof}
The proof is done inductively as follows:
{\small 
\begin{eqnarray*}
[\det(M_{g})]&=&\left[-\frac{x_1^{a_1-2}x_{2}}{g_1-1} \left( \frac{1-g_1^{a_1}}{1-g_1} - a_1 \right)
\left(\prod_{i=2}^{d_g}\frac{-a_i}{1-g_i}\right)x_2^{a_2-2}x_3^{a_3-1}\cdots x_{d_g}^{a_{d_g}-1}x_{d_g+1}\right]\\
& &+\left[x_1^{2a_1-2} \frac{g_1^{a_1}}{(1 - g_1)^2}\left(\prod_{i=3}^{d_g}\frac{-a_i}{1-g_i}\right)x_3^{a_3-2}x_4^{a_4-1}\cdots x_{d_g}^{a_{d_g}-1}x_{d_g+1}\right]\\
&=&\left(\prod_{i=1}^{d_g}\frac{-a_i}{1-g_i}\right)[x_1^{a_1-2}x_2^{a_2-1}\cdots x_{d_g}^{a_{d_g}-1}x_{d_g+1}]\\
& &+\left(-\frac{1-g_1^{a_1}}{1-g_2}-g_1^{a_1}\right)\frac{a_2}{(1 - g_1)^2} \left(\prod_{i=3}^{d_g}\frac{-a_i}{1-g_i}\right)
[x_1^{a_1-2}x_2^{a_2-1}\cdots x_{d_g}^{a_{d_g}-1}x_{d_g+1}]\\
&=&\left(\prod_{i=1}^{d_g}\frac{-a_i}{1-g_i}\right)[x_1^{a_1-2}x_2^{a_2-1}\cdots x_{d_g}^{a_{d_g}-1}x_{d_g+1}],
\end{eqnarray*}}
where we used the relations $[x_1^{a_1}]=-a_2[x_2^{a_2-1}x_3]$ and $g_1^{a_1}g_2=1$.
Since the element $[H_{g,g^{-1}}]$ is easily calculated and it is given by 
$(\prod_{i=1}^{d_g}a_i)[x_1^{a_1-2}x_2^{a_2-1}\cdots x_{d_g}^{a_{d_g}-1}x_{d_g+1}]$ (cf. \cite[Proof of Lemma~29]{BTW16}), 
the statement follows.
\end{proof}
Thus we have finished the proof of Proposition~\ref{prop:10} for $f$ of chain type.
\subsection{Loop type $f = x_1^{a_1}x_2 + \dots + x_n^{a_n}x_1$}\label{sec:loop}
First, we list the ingredients for the product formula \eqref{eq: HH cup}: 
{\small 
\begin{align*}
 \rmH_f(\bx,\by,\bz) &= \sum_{i=1}^{n-1} \frac{x_{i+1}}{y_i-z_i} \left( \frac{x_i^{a_i}-y_i^{a_i}}{x_i-y_i} - \frac{x_i^{a_i}-z_i^{a_i}}{x_i-z_i} \right)\theta_i \otimes\theta_i + \sum_{i=1}^{n-1} \frac{y_i^{a_i}-z_i^{a_i}}{y_i-z_i} \theta_{i+1} \otimes\theta_{i}
 \\
 & + \frac{z_1}{y_n-z_n} \left( \frac{x_n^{a_n}-y_n^{a_n}}{x_n-y_n} - \frac{x_n^{a_n}-z_n^{a_n}}{x_n-z_n} \right)\theta_n \otimes\theta_n
   + \frac{x_n^{a_n}-y_n^{a_n}}{x_n-y_n} \theta_{n} \otimes\theta_{1}.
\end{align*}
\[
\rmH_{f,g}(\bx) = \sum_{i=1}^{n-1} x_i^{a_i-1} \frac{g_i^{a_i}}{1 - g_i} \theta_i \theta_{i+1} + x_n^{a_n-1} \frac{1}{g_n -1} \theta_1 \theta_n, \quad g=(g_1,\dots,g_n)\in G_f\backslash\{\id\}.
\]}
They give
{\small 
\begin{multline*}
\rmH_f(\bx,g(\bx),\bx) 
= \sum_{i=1}^{n-1} \frac{x_i^{a_i-2}x_{i+1}}{g_i-1} \left( \frac{1-g_i^{a_i}}{1-g_i} - a_i \right)\theta_i \otimes\theta_i
+ \sum_{i=1}^{n-1} \frac{g_i^{a_i}-1}{g_i-1}x_i^{a_i-1} \theta_{i+1} \otimes\theta_{i}\\
+ \frac{x_n^{a_n-2}x_{1}}{g_n-1} \left( \frac{1-g_n^{a_n}}{1-g_n} - a_n \right)\theta_n \otimes\theta_n
+ \frac{g_n^{a_n}-1}{g_n-1}x_n^{a_n-1} \theta_n \otimes\theta_{1}.
\end{multline*}
\[
\rmH_{f,g^{-1}}(g(\bx)) = \sum_{i=1}^{n-1} x_i^{a_i-1} \frac{1}{g_i-1} \theta_i \theta_{i+1} + x_n^{a_n-1} \frac{g_n^{a_n}}{1-g_n} \theta_1 \theta_n,\quad g=(g_1,\dots,g_n)\in G_f\backslash\{\id\}.
\]}
Consider a square matrix $M_g$ of size $n$ whose $(i,j)$-th entry is given by
{\small 
\begin{multline}
(M_g)_{ij} := -\frac{x_i^{a_i-2}x_{i+1}}{g_i-1} \left( \frac{1-g_i^{a_i}}{1-g_i} - a_i \right)\delta_{i,j} 
+ x_i^{a_i-1} \frac{g_i^{a_i}}{1 - g_i} \delta_{i+1,j}+ x_{i-1}^{a_{i-1}-1} \frac{1}{g_{i-1} - 1}\delta_{i-1,j}\\
+ x_n^{a_n-1} \frac{g_n^{a_n}}{1 - g_n} \delta_{i,n}\delta_{j,1}+ x_n^{a_n-1} \frac{1}{g_n - 1}\delta_{i,1}\delta_{j,n}.
\end{multline}}
\begin{lemma}\label{lem:13}
For each $g\in G_f\backslash\{\id\}$, we have the following equality in $\Jac(f)$:
\begin{equation}
\sigma_{g,g^{-1}}=(-1)^{\frac{n(n-1)}{2}}\left[\det(M_{g})\right].
\end{equation}
\end{lemma}
\begin{proof}
The proof is similar to that of Lemma~\ref{lem:11}. The only differences are the term $\theta_n \otimes \theta_1$ in $\rmH_f(\bx,g(\bx),\bx)$ and $\theta_1 \theta_n$ in
$\rmH_{f,g}(\bx)$ (and also in $\rmH_{f,g^{-1}}(g(\bx))$) that give some additional contributions to $\sigma_{g,g^{-1}}$ (in particular the third and fourth summands in proof of Lemma~\ref{lem:14} below).
\end{proof}
\begin{lemma}\label{lem:14}
For each $g\in G_f\backslash\{\id\}$, we have 
{\small
\begin{equation}
\left[\det(M_{g})\right]=\left(\prod_{i=1}^n\frac{1}{g_i-1}\right)\left(\left(\prod_{l=1}^{n} a_l\right)-(-1)^n\right)[x_1^{a_1-1}\cdots x_n^{a_n-1}]
=\left(\prod_{i=1}^n\frac{1}{g_i-1}\right) [H_{g,g^{-1}}].
\end{equation}
}
\end{lemma}
\begin{proof}
First, we have the following
\begin{lemma}
Let $M_g'$ be a square matrix of size $n$ whose $(i,j)$-th entry is given by
{\small 
\[
(M_g')_{ij} := -\frac{x_i^{a_i-2}x_{i+1}}{g_i-1} \left( \frac{1-g_i^{a_i}}{1-g_i} - a_i \right)\delta_{i,j} 
+ x_i^{a_i-1} \frac{g_i^{a_i}}{1 - g_i} \delta_{i+1,j}+ x_{i-1}^{a_{i-1}-1} \frac{1}{g_{i-1} - 1}\delta_{i-1,j}.
\]}
We have 
{\small 
\[
\left[\det(M_{g}')\right]=\left(\prod_{i=1}^n\frac{1}{g_i-1}\right)
\left(\sum_{k=0}^n (-1)^k\left(\prod_{l=1}^{n-k} a_l\right)\left(\prod_{l=n-k+1}^{n} \frac{1-g_l^{a_l}}{1-g_l}\right)\right)[x_1^{a_1-1}\cdots x_n^{a_n-1}] .
\]}
\end{lemma}
\begin{proof}
One can show easily the statement by induction where we use the relations 
$[x_1^{a_1}]=-a_{2}[x_{2}^{a_2-1}x_3]$, $\dots$, $[x_{n-1}^{a_{n-1}}]=-a_{n}[x_{n}^{a_n-1}x_1]$
and $g_1^{a_1}g_{2}=\dots =g_{n-1}^{a_{n-1}}g_{n}=1$. 
\end{proof}
The statement follows from a direct calculation of the determinant:
{\small 
\begin{eqnarray*}
& &\left(\prod_{i=1}^n \left(g_i-1\right)\right)[\det(M_{g})]\\
&=&\left(\sum_{k=0}^n (-1)^k\left(\prod_{l=1}^{n-k} a_l\right)\left(\prod_{l=n-k+1}^{n} \frac{1-g_l^{a_l}}{1-g_l}\right)\right)[x_1^{a_1-1}\cdots x_n^{a_n-1}]\\
& &+a_1\frac{1-g_n^{a_n}}{1-g_n}\left(\sum_{k=0}^{n-2} (-1)^k\left(\prod_{l=2}^{n-k-1} a_l\right)\left(\prod_{l=n-k}^{n-1} \frac{1-g_l^{a_l}}{1-g_l}\right)\right)[x_1^{a_1-1}\cdots x_n^{a_n-1}]\\
& &+(-1)^{n-1}[x_1^{a_1-1}\cdots x_n^{a_n-1}]-\left(\prod_{i=1}^n g_i^{a_i}\right)[x_1^{a_1-1}\cdots x_n^{a_n-1}]\\
&=&\left(\sum_{k=0}^n (-1)^k\left(\prod_{l=1}^{n-k} a_l\right)\left(\prod_{l=n-k+1}^{n} \frac{1-g_l^{a_l}}{1-g_l}\right)\right)[x_1^{a_1-1}\cdots x_n^{a_n-1}]\\
& &-\left(\sum_{k=1}^{n-1} (-1)^k\left(\prod_{l=1}^{n-k} a_l\right)\left(\prod_{l=n-k+1}^{n} \frac{1-g_l^{a_l}}{1-g_l}\right)\right)[x_1^{a_1-1}\cdots x_n^{a_n-1}]\\
& &-(-1)^n[x_1^{a_1-1}\cdots x_n^{a_n-1}]-\left(\prod_{i=1}^n g_i^{a_i}\right)[x_1^{a_1-1}\cdots x_n^{a_n-1}]\\
&=& \left(a_1\cdots a_n-(-1)^n\right)[x_1^{a_1-1}\cdots x_n^{a_n-1}],
\end{eqnarray*}}
where we used the relations $[x_{n}^{a_{n}}]=-a_{1}[x_{1}^{a_1-1}x_2]$, 
$1-g_1^{a_1}=-g_{2}^{-1}(1-g_{2})$, $\dots$,  $1-g_{n-1}^{a_{n-1}}=-g_{n}^{-1}(1-g_{n})$, 
$1-g_n^{a_n}=g_1^{-1}(1-g_{1})$ and $g_1^{a_1}\cdots g_n^{a_n}=g_1^{-1}\cdots g_n^{-1}$.
Since the element $[H_{g,g^{-1}}]$ is nothing but $[\hess(f)]$ and it is given by 
$(a_1\cdots a_n-(-1)^n)[x_1^{a_1-1}\cdots x_n^{a_n-1}]$.
\end{proof}
Thus we have finished the proof of Proposition~\ref{prop:10} also for $f$ of loop type.

\section{Categorical equivalence}
In our previous papers \cite{BTW16,BTW17}, we found an algebra isomorphism
\begin{equation}\label{eq:34}
\Jac(\overline f)=\Jac(\overline f,\{\id\}) \cong \Jac(f,G),
\end{equation}
for some invertible polynomials $f$, $\overline f$ and subgroups $G\subseteq G_{f}\cap \SL(3,\CC)$.
More precisely, $f$ is an invertible polynomial defining an ADE singularity or an exceptional unimodal singularity
and $G$ is any subgroup of $G_{f}\cap \SL(3,\CC)$. 
The polynomial $\overline f$ is defined as the restriction of the map $\widehat f:\widehat{\CC^3/G}\longrightarrow \CC$ to 
a chart $U$ isomorphic to $\CC^3$ containing all the critical points of $\widehat f$ where $\widehat{\CC^3/G}$ is 
a crepant resolution of $\CC^3/G$. See \cite[Theorem~63]{BTW16} and \cite[Theorem~1]{BTW17}\footnote{Theorem~1 in \cite{BTW17} is stated in a different way. Namely, we used as $\overline f$ the Berglund--H\"ubsch transpose of $f$.
However, it is easy to check that two different $\overline f$'s give isomorphic Jacobian algebras.}. 

On the other hand, one may ask whether the (quasi-)equivalence of categories of matrix factorizations holds:
\begin{equation}\label{eq:35}
\MF(\overline f)=\MF_{\{\id\}}(\overline f)\cong \MF_G(f).
\end{equation}
This is true\footnote{the second-named author thanks M.~Wymess for 
reminding him of this at the workshop ``Matrix factorizations and related topics'' at ICMS, 24--28 July 2017.}
 from the construction of $\overline f$ due to the local property of the category of singularities  
of $\{\widehat{f}=0\}\subset \widehat{\CC^3/G}$ which is equivalent to $\MF_G(f)$ \cite[Proposition~1.14]{Or}. 
It is worth mentioning that in \cite{CCR16, CN16} the equivalence is given as their ``orbifold equivalence'' $\overline f\sim_{orb} f$.

The Hochschild cohomology and the cup product is invariant under categorical equivalences.
Theorem~\ref{theorem: main} of this paper confirms the compatibility of the isomorphism \eqref{eq:34} and 
the equivalence \eqref{eq:35} as expected.


\begin{thebibliography}{99}
\bibitem[AGV85]{AGV85} 
	V. Arnold, A. Gusein-Zade, A. Varchenko, 
	\emph{ Singularities of Differentiable Maps}, vol I 
	\\Monographs in Mathematics, 82. Birkh\"auser Boston, Inc., Boston, MA, 1985
\bibitem[BTW16]{BTW16}
	A.~Basalaev, A.~Takahashi, E.~Werner, \emph{Orbifold Jacobian algebras for invertible polynomials}, arXiv preprint: 1608.08962.
\bibitem[BTW17]{BTW17}
	A.~Basalaev, A.~Takahashi, E.~Werner, \emph{Orbifold Jacobian algebras for exceptional unimodal singularities}, Arnold Math J. (2017). https://doi.org/10.1007/s40598-017-0076-8.
\bibitem[CRCR16]{CCR16}
	N.~Carqueville, A.R.~Camacho, I.~Runkel, \emph{Orbifold equivalent potentials}, Journal of Pure and Applied Algebra, 220(2), (2016). 759–781. http://doi.org/10.1016/j.jpaa.2015.07.015
\bibitem[ET13]{ET13}
	W.~Ebeling, A.~Takahashi,
	\emph{ Variance of the exponents of orbifold Landau--Ginzburg models}, Math. Res. Lett. {\bf 20} (1) (2013), 51--65.
\bibitem[HLL]{HLL}
	W.~He, S.~Li, Y.~Li, \emph{G-twisted braces and orbifold Landau-Ginzburg Models}, arXiv:1801.04560.
\bibitem[KS]{KS} 
	M. Kreuzer, H. Skarke: \emph{On the classification of quasihomogeneous functions}, Commun. Math. Phys. 150, 137--147 (1992).
\bibitem[Or]{Or}
	D.~Orlov, \emph{Triangulated categories of singularities and D-branes in Landau-Ginzburg models}, Proc. Steklov Inst. Math. 2004, no. 3(246), 227--248.
\bibitem[RCN16]{CN16}
	A.R.~Camacho, R.~Newton, \emph{Orbifold autoequivalent exceptional unimodal singularities}, arXiv preprint: 1607.07081
\bibitem[S17]{S17}
	D.~Shklyarov, \emph{On Hochschild invariants of Landau--Ginzburg orbifolds},
	arXiv preprint: 1708.06030v1.
\end{thebibliography}
\end{document}